\documentclass[10pt]{amsart}
\usepackage{amsmath, amssymb}
 \usepackage{mathrsfs}
\newcommand{\no}[1]{#1}
\renewcommand{\no}[1]{}
\no{\usepackage{times}\usepackage[subscriptcorrection, slantedGreek, nofontinfo]{mtpro}
\renewcommand{\Delta}{\upDelta}}
\usepackage{color}

\date{\today}
\newtheorem{theorem}{Theorem}[section]

\newtheorem{lemma}{Lemma}[section]

\newtheorem{corollary}{Corollary}[section]

\newtheorem{remark}{Remark}[section]

\numberwithin{equation}{section}

\title[Coefficient identification in terms of the resolvent]{Lipschitz dependence of the coefficients on the resolvent and greedy approximation for scalar elliptic problems}

\thanks{
The second author was partially supported
by the ERC Advanced Grant DYCON (Dynamic Control), the ANR (France) Project ICON, 
Grants FA9550-14-1-0214 of the EOARD-AFOSR, FA9550-15-1-0027 of AFOSR, and
 MTM2014-52347 
 of the MINECO (Spain).
 }
 
\author[Mourad Choulli]{Mourad Choulli\dag}
\address{\dag IECL, UMR CNRS 7502, Universit\'e de Lorraine, Boulevard des Aiguillettes BP 70239 54506 Vandoeuvre Les Nancy cedex- Ile du Saulcy - 57 045 Metz Cedex 01 France}
\email{mourad.choulli@univ-lorraine.fr}

\author[Enrique Zuazua]{Enrique Zuazua\ddag}
\address{\ddag 
1. DeustoTech - Fundación Deusto, Avda Universidades, 24, 48007, Bilbao - Basque Country - Spain
\vskip.1cm
2. Departamento de Matemáticas, Universidad Autónoma de Madrid, 28049 Madrid - Spain
\vskip.1cm
3. Facultad Ingeniería, Universidad de Deusto, Avda. Universidades, 24, 48007, - Basque Country - Spain
\vskip.1cm
4. Sorbonne Universit\'es, UPMC Univ Paris 06, CNRS UMR 7598, Laboratoire Jacques-Louis Lions, F-75005, Paris, France}
\email{enrique.zuazua@uam.es}

\date{\today}

\begin{document}

\begin{abstract}
We analyze the inverse problem of identifying the  diffusivity coefficient of a scalar elliptic equation as a function of the resolvent operator. We prove that, within the class of measurable coefficients, bounded above and below by positive constants, the resolvent determines the diffusivity in an unique manner. Furthermore we prove that the inverse mapping from resolvent to the coefficient is Lipschitz in suitable topologies.  This result plays a key role when applying greedy algorithms to the approximation of parameter-dependent elliptic problems in an uniform and robust manner, independent of the given source terms. In one space dimension the results can be improved using the  explicit expression of solutions, which allows to link distances between one resolvent and a linear combination of finitely many others and the corresponding distances on coefficients. These results are also extended to multi-dimensional  elliptic equations with variable density coefficients. We also point out towards some possible extensions and open problems.

\medskip
\noindent
{\sc R\'esum\'e.} Nous examinons le probl\`eme inverse d'identifier le coefficient de diffusion comme fonction de la r\'esolvante pour des \'equations elliptiques scalaires. Nous \'etablissons, pour des topologies appropri\'ees,  un r\'esultat de stabilit\'e Lipschitz pour une classe de coefficients de diffusion mesurables, minor\'es et major\'es par des constantes positives fix\'ees a priori. Ce r\'esultat intervient de mani\`ere essentielle dans le d\'eveloppement d'algorithmes  {\it greedy} pour l'approximation d'une famille param\'etr\'ee de probl\`emes elliptiques de mani\`ere robuste et uniforme par rapport au terme source. Nous traitons s\'epar\'ement le cas de la dimension un pour lequel nous disposons de formules explicites de repr\'esentation des solutions permettant de comparer la distance entre une r\'esolvante et une combinaison lin\'eaire d'un nombre fini d'autres et des  coefficients correspondants, et un d\'evelopemment complet de l'approche {\it greedy}. Nous \'etendons ces r\'esultats  au probl\`eme de l'identification de la densit\'e \`a partir de l'op\'erateur r\'esolvant correspondant. Nous signalons aussi quelques probl\`emes ouverts, en particulier dans le cas multi-dimensionnel.


\end{abstract}

\maketitle


\section{Introduction}
Let $\Omega$ be a  bounded domain of $\mathbb{R}^n$, $n\ge 1$. Fix $0<\sigma _0<\sigma _1$ and consider the class of scalar diffusivity coefficients
\[
\Sigma =\{ \sigma \in L^\infty (\Omega );\; \sigma _0\le \sigma \le  \sigma _1\; \mbox{a.e. in}\; \Omega \}.
\]

\smallskip
In the sequel $H_0^1(\Omega )$ is endowed with the norm
\[
\|w\|_{H_0^1(\Omega )}=\|\nabla w\|_{L^2(\Omega )^n}.
\]

\smallskip
For $\sigma \in \Sigma$, let $A_\sigma : H^1_0(\Omega )\rightarrow H^{-1}(\Omega )$ be the bounded operator given by 
\[
A_\sigma u=-\mbox{div}(\sigma \nabla u).
\]
The inverse or resolvent operator $R_\sigma$ maps continuously $H^{-1}(\Omega)$ into $H^1_0(\Omega)$. 

\smallskip
To be more precise, denote by $\langle \cdot ,\cdot \rangle _{-1,1}$ the duality pairing between $H^{-1}(\Omega )$ and $H_0^1(\Omega )$. For $f\in H^{-1}(\Omega )$, consider the variational problem of finding $w\in H_0^1(\Omega )$ so that
\begin{equation}\label{vp}
\int_\Omega \sigma \nabla w\cdot \nabla v=\langle f ,v\rangle _{-1,1}\;\mbox{for any}\;  v\in H_0^1(\Omega ).
\end{equation}
According to Lax-Milgram's lemma,  \eqref{vp} has a unique solution $u_\sigma \in H_0^1(\Omega )$. Moreover, the energy estimate yields
\begin{equation}\label{LML}
\|u_\sigma \|_{H_0^1(\Omega )}\le \sigma_0^{-1} \|f\|_{H^{-1}(\Omega )}.
\end{equation}

Indeed, using the solution itself $u_\sigma$ as test function we have
$$
\sigma_0 \int_\Omega  |\nabla u_\sigma|^2 dx  \le \int_\Omega \sigma(x) |\nabla u_\sigma|^2 dx = \langle f, u\rangle_{-1, 1} \le ||f||_{H^{-1}(\Omega)} ||\nabla u_\sigma||_{L^2(\Omega)^n}
$$
and consequently
$$
\sigma_0 ||\nabla u_\sigma||_{L^2(\Omega)^n}
 \le ||f||_{H^{-1}(\Omega)}.
  $$

As we have seen, the coefficient $\sigma$ of the elliptic equation determines uniquely the resolvent operator $R_\sigma$. We address the inverse problem consisting on identifying the coefficient $\sigma$ in terms of the resolvent $R_\sigma$. 

The main result of this paper ensures the Lipschitz character of the inverse map in suitable topologies:

\begin{theorem}\label{theorem1}
For any $\sigma ,\widetilde{\sigma}\in \Sigma$,
\begin{equation}\label{Lest}
\sigma_0^2\|R_\sigma -R_{\widetilde{\sigma}}\|_{-1,1}\le \|\sigma -\widetilde{\sigma} \|_{L^\infty (\Omega )}\le \sigma_1^2\|R_\sigma -R_{\widetilde{\sigma}}\|_{-1,1}. 
\end{equation}
Here and henceforth $\|\cdot \|_{-1,1}$ denotes the norm in $\mathscr{B}(H^{-1}(\Omega ),H_0^1(\Omega ))$.
\end{theorem}

As observed by Albert Cohen (\cite{cohen}) this result can be easily extended to the case of continuous matrix valued coefficients. The proof is the same as the one we shall develop but using test functions that are scaled in a more pronounced manner in a distinguished direction. The extension of this result to the general case of measurable matrix valued coefficients seems however more delicate and requires further work. See Remark \ref{remarkcohen} below.

As we shall see below this question and result arise in the context of  parameter-dependent elliptic equations and it is of potential use (but not sufficient) to develop greedy algorithms to build fast and efficient approximation methods.

\smallskip
By inspecting the proof one can see that  Theorem \ref{theorem1}  holds for any domain $\Omega$  for which Poincar\'e's inequality holds.

\smallskip
Let $d_\infty$ be the distance (between diffusivity coefficients)  induced by the $L^\infty$-norm. Inequality \eqref{Lest} in Theorem \ref{theorem1} can rephrased as
\[
\sigma_0^2 d_R\le d_\infty \le \sigma_1^2 d_R\;\; \mbox{on}\; \Sigma ,
\]
where $d_R$ is the metric on $\Sigma$ defined as follows
\[
d_R(\sigma ,\widetilde{\sigma})=\|R_\sigma -R_{\widetilde{\sigma}}\|_{-1,1},\;\; \sigma ,\widetilde{\sigma}\in \Sigma .
\]

\smallskip
The rest of the paper is organized as follows. We prove Theorem \ref{theorem1} in Section 2. In Section 3 we adapt the proof of Theorem \ref{theorem1} in order to establish a Lipschitz stability estimate in the case of the Neumann or Robin boundary conditions. The case of a BVP with non homogenous boundary values  is treated in Section 4. In that case the resolvent is obtained by varying the boundary data. Due to the  elliptic smoothing effect, which prevents the information to propagate completely from the boundary to the interior, in the present case we are only able to prove H\"older stability. We devote Section 5 to the one dimensional case. Taking advantage of the explicit representation formula for the solution of the BVP we establish a Lipschitz stability property and also estimate the distance from a resolvent  to the linear subspace generated by a finite number of them.
In Section 6 we present the motivation of this paper in the context of greedy algorithms for parameter-dependent elliptic equations and we fully develop it in the one-dimensional case, using the results of Section 5. We also added a remark in section 6 in order to explain how the same program can be fully developed, in any dimension, for elliptic equations with variable density coefficients. We close with Section 7 devoted to some open problems, and in particular to the  extension of the greedy approach to the multi-dimensional diffusivity  model.

\section{Proof of Theorem \ref{theorem1}}

The first inequality in \eqref{Lest} is contained in the following elementary lemma. 

\begin{lemma}\label{lemma1}
For any $\sigma ,\widetilde{\sigma}\in L^\infty (\Omega )$ satisfying $\sigma _0\le \sigma ,\widetilde{\sigma}$,
\[
\|R_\sigma -R_{\widetilde{\sigma}}\|_{-1,1}\le  \sigma _0^{-2}\|\sigma -\widetilde{\sigma} \|_{L^\infty (\Omega  )}.
\]
\end{lemma}

\begin{proof}
From \eqref{vp} we have
\[
\int_\Omega \sigma \nabla u_\sigma \cdot \nabla vdx=\int_\Omega \widetilde{\sigma} \nabla u_{\widetilde{\sigma}} \cdot \nabla vdx\;\;\mbox{for any}\;  v\in H_0^1(\Omega ).
\]
Hence
\[
\int_\Omega \sigma \nabla (u_\sigma -u_{\widetilde{\sigma}})\cdot \nabla vdx=\int_\Omega (\widetilde{\sigma}-\sigma ) \nabla u_{\widetilde{\sigma}} \cdot \nabla vdx\;\;\mbox{for any}\;  v\in H_0^1(\Omega ).
\]
The particular choice of $v=u_\sigma -u_{\widetilde{\sigma}}$ in the identity above, thanks to (\ref{LML}),  yields
\begin{align}\label{basicest}
\sigma _0\int_\Omega |\nabla (u_\sigma -u_{\widetilde{\sigma}})|^2 dx\le \int_\Omega \sigma |\nabla (u_\sigma -u_{\widetilde{\sigma}})|^2 dx&\le \|\widetilde{\sigma}-\sigma\|_{L^\infty (\Omega )}\|\nabla  u_{\widetilde{\sigma}}\|_{L^2(\Omega )^n}\|\nabla  (u_\sigma -u_{\widetilde{\sigma}})\|_{L^2(\Omega )^n}
\\
&\le  \sigma _0^{-1}\|\widetilde{\sigma}-\sigma\|_{L^\infty (\Omega )}\|f\|_{H^{-1}(\Omega )}\|\nabla  (u_\sigma -u_{\widetilde{\sigma}})\|_{L^2(\Omega )^n}.\nonumber
\end{align}

From \eqref{basicest} we deduce immediately the expected inequality.
\end{proof}


Next, we establish the key lemma that we will use to prove the second inequality in \eqref{Lest}.

\begin{lemma}\label{lemma0.1}
Let $\gamma \in L^\infty (\Omega )$. For a.e. $x_0\in \Omega$, there exists a sequence $(u_{x_0,\epsilon} )$ in  $H_0^1(\Omega )$ so that $\|u_{x_0,\epsilon} \|_{H_0^1(\Omega )}=1$, for each $\epsilon$, and
\[
\lim_\epsilon\int_\Omega \gamma (x) |\nabla u_{x_0,\epsilon} |^2dx=\gamma (x_0).
\]
\end{lemma}

\begin{proof}
Let $\epsilon >0$ and set
\[
\varphi _\epsilon (r)=0,\; r\le 0,\;\; \varphi _\epsilon (r)=r,\; 0<r<\epsilon ,\;\; \varphi _\epsilon (r)=\epsilon ,\; r\ge \epsilon ,
\]
the continuous function such that $\varphi_\epsilon '=\chi_{(0,\epsilon)}$, where $\chi_{(0,\epsilon)}$ is the characteristic function of the interval $(0,\epsilon )$.

\smallskip
Let $x_0\in \Omega$ and $\epsilon _0$ be sufficiently small is such a away that $B(x_0,\epsilon _0)\subset \Omega$. Define $u_{x_0,\epsilon}$, $0<\epsilon \le \epsilon _0$ by
\[
u_{x_0,\epsilon}(x)=\frac{1}{\sqrt{|B(x_0,\epsilon )|}}\varphi _\epsilon (|x-x_0| ).
\]
It is easy to see that
\[
\nabla u_{x_0,\epsilon} (x)=\chi_{B(x_0, \epsilon)}(x)\frac{x-x_0}{|x-x_0|\sqrt{|B(x_0,\epsilon )}|}
\]
and that $u_{x_0,\epsilon}$ belongs to $H_0^1(\Omega )$.
Whence 
\[
|\nabla u_{x_0,\epsilon}|^2= \frac{1}{|B(x_0,\epsilon )|}\chi _{B(x_0,\epsilon )}.
\]
Therefore $\|u_{x_0,\epsilon}\|_{H_0^1(\Omega )}=1$ and by Lebesgue's differentiation theorem
\[
\int_\Omega \gamma (x)|\nabla u_{x_0,\epsilon}|^2dx =\frac{1}{|B(x_0,\epsilon )|}\int_{B(x_0, \epsilon)}\gamma (x)dx \underset{\epsilon\rightarrow 0}{\longrightarrow} \gamma (x_0) \;\; \mbox{a.e.}\; x_0\in \Omega .
\]
\end{proof}

\begin{corollary}\label{corollary0.1}
Let $\gamma \in L^\infty (\Omega )$ be so that
\begin{equation}\label{0.1.1hyp}
\int_\Omega \pm \gamma |\nabla u|^2dx \le C,\;\; \mbox{for any}\; u\in H_0^1(\Omega ),\; \|u\|_{H_0^1(\Omega )}=1,
\end{equation}
for some constant $C>0$. Then
\begin{equation}\label{0.1.1}
\|\gamma \|_{L^\infty (\Omega )}\le C.
\end{equation}
\end{corollary}
\begin{proof}
In light of Lemma \ref{lemma0.1}, for a.e. $x_0\in \Omega $, there exists a sequence $(u_n^\pm)$ in  $H_0^1(\Omega )$ so that $\|u_n^\pm\|_{H_0^1(\Omega )}=1$, for each $n$, and
\[
\lim_n\int_\Omega \pm \gamma (x) |\nabla u_n^\pm|^2dx=\pm \gamma (x_0).
\]
Therefore, in view of (\ref{0.1.1hyp}), $|\gamma (x_0)|\le C$ a.e. $x_0\in \Omega$, implying \eqref{0.1.1}.
\end{proof}

We are now ready to complete the proof of Theorem \ref{theorem1}. Fix $\sigma ,\widetilde{\sigma}\in \Sigma_0$. Starting from the identity
\[
A_\sigma -A_{\widetilde{\sigma}}= A_\sigma(R_{\widetilde{\sigma}}-R_\sigma )A_{\widetilde{\sigma}},
\]
we get
\begin{equation}\label{0.1}
\| A_\sigma -A_{\widetilde{\sigma}}\|_{1,-1}\le \sigma _1^2 \| R_\sigma -R_{\widetilde{\sigma}}\|_{-1,1}.
\end{equation}
On the other hand
\[
\langle  (A_\sigma -A_{\widetilde{\sigma}})u, v \rangle _{-1,1}= \int_\Omega (\sigma -\widetilde{\sigma})\nabla u \cdot \nabla v dx,\;\; u,v\in H_0^1(\Omega ),
\]
implying 
\[
\int_\Omega (\sigma -\widetilde{\sigma})\nabla u \cdot \nabla v dx\le \| A_\sigma -A_{\widetilde{\sigma}}\|_{1,-1} \|u\|_{H_0^1(\Omega )}\|v \|_{H_0^1(\Omega )},\;\; u,v \in H_0^1(\Omega ).
\]
Combined with \eqref{0.1}, this estimate yields
\begin{equation}\label{0.2}
\int_\Omega \gamma \nabla u \cdot \nabla v dx\le \sigma _1^2 \| R_\sigma -R_{\widetilde{\sigma}}\|_{-1,1}\|u\|_{H_0^1(\Omega )}\|v \|_{H_0^1(\Omega )},\;\; u,v\in H_0^1(\Omega ),
\end{equation}
where we set $\gamma =\sigma -\widetilde{\sigma}$. Hence
\begin{equation}\label{0.3}
\int_\Omega \gamma |\nabla u |^2 dx\le \sigma _1^2 \| R_\sigma -R_{\widetilde{\sigma}}\|_{-1,1},\;\; u\in H_0^1(\Omega ),\; \|u\|_{H_0^1(\Omega )}=1.
\end{equation}
But, by symmetry, \eqref{0.3} holds when $\gamma$ is substituted by $-\gamma=\widetilde{\sigma}- \sigma$. That is we have
\[
\int_\Omega \pm \gamma |\nabla u |^2 dx\le \sigma _1^2 \| R_\sigma -R_{\widetilde{\sigma}}\|_{-1,1},\;\; u\in H_0^1(\Omega ),\; \|u\|_{H_0^1(\Omega )}=1
\]
which, by Corollary \ref{corollary0.1},  yields the second inequality of \eqref{Lest}.

\begin{remark}\label{remarkcohen}
{\rm
Following interesting discussions with A. Cohen \cite{cohen}, here we present possible extensions to the anisotropic case that can be obtained following the method of proof of Theorem \ref{theorem1}.

\smallskip
(i) Consider, to begin with, in dimension two, the case of an anisotropic diagonal conductivity
 \[
 \sigma =\mbox{diag}( \sigma _1,\sigma_2).
 \]
 In the sequel, for simplicity sake, we identify $\sigma$ with $(\sigma _1,\sigma _2)$.
 
 \smallskip
 Fix $0<a _0<a _1$ and let $\Sigma'$ the set of $\sigma =(\sigma_1,\sigma _2)\in C(\overline{\Omega})\oplus C(\overline{\Omega})$ satisfying
 \[
a _0|\xi |^2\le \sigma _1(x)\xi_1^2+\sigma_2(x)\xi_2^2\;\; \mbox{and}\;\; \sigma _1(x),\sigma_2(x)\le a_1 \;\; \mbox{for any}\; x\in \Omega,\; \xi \in \mathbb{R}^2.
 \]
Let $\sigma ,\widetilde{\sigma}\in \Sigma '$.  With similar notations ($\gamma$ denotes the difference of two diffusivity pairs), instead of \eqref{0.3} we have in the present case
 \begin{equation}\label{0.4}
\int_\Omega \left[\gamma_1 (\partial _1 u) ^2 + \gamma_2 (\partial _2 u) ^2\right]dx\le a _1^2 \| R_\sigma -R_{\widetilde{\sigma}}\|_{-1,1},\;\; u\in H_0^1(\Omega ),\; \|u\|_{H_0^1(\Omega )}=1.
\end{equation}

Without loss of generality, we can assume that 
\[
\|\gamma _1\|_{C(\overline{\Omega})}=\max \left(\|\gamma _1\|_{C(\overline{\Omega})},\|\gamma _2\|_{C(\overline{\Omega})}\right)
\]

Fix $0<\delta <1$ and let $x_0\in \Omega$ so that
\[
|\gamma_1(x_0)|=(1-\delta )\|\gamma _1\|_{C(\overline{\Omega})}.
\]
Substituting $\sigma -\widetilde{\sigma}$ by $\widetilde{\sigma}-\sigma$, we can always assume that $\gamma _1(x_0)>0$.

\smallskip

Let $\varphi _\epsilon$ be as in Lemma \ref{lemma0.1}, $D_K(x_0,\epsilon )=\{(x_1,x_2);\; |(K(x_1 -x_{0,1}),x_2-x_{0,2})|\le \epsilon\}$ and consider the test function
\[
\varphi_{x_0,\epsilon}(x_1,x_2)=\frac{1}{\sqrt{D_K(x_0,\epsilon )}} \varphi_\epsilon (|(K(x_1 -x_{0,1}),x_2-x_{0,2})|),
\]
where the scaling parameter $K$ is chosen in such a away that $K\gamma _1(x_0)-|\gamma _2(x_0)|\ge \gamma _1(x_0)$, and $\epsilon$ is sufficiently small so that $\mbox{supp}(\varphi_{x_0,\epsilon}) \Subset \Omega$.

\smallskip
Define $\psi_{x_0,\epsilon} $ by $\psi_{x_0,\epsilon} (x_1,x_2)=\varphi_{x_0,\epsilon}(x_2,x_1)$ and observe that we still have $\mbox{supp}(\psi_{x_0,\epsilon}) \Subset \Omega$ provided that $\epsilon$ is sufficiently small. Then \eqref{0.4} with $u=\varphi_{x_0,\epsilon}$ and $u=\psi_{x_0,\epsilon}$ successively  yields in a straightforward manner
\[
\frac{1}{|D_K(x_0,\epsilon )|}\int_{D_K(x_0,\epsilon )} \left( K\gamma _1+\gamma _2\right)dx \le 2a _1^2 \| R_\sigma -R_{\widetilde{\sigma}}\|_{-1,1}.
\]
Whence
\[
\frac{1}{|D_K(x_0,\epsilon )|}\int_{D_K(x_0,\epsilon )} \gamma _1dx \le 2a _1^2 \| R_\sigma -R_{\widetilde{\sigma}}\|_{-1,1}.
\]
A standard continuity argument leads
\[
(1-\delta )\max \left(\|\gamma _1\|_{C(\overline{\Omega})},\|\gamma _2\|_{C(\overline{\Omega})}\right)=\gamma _1(x_0)\le 2a _1^2 \| R_\sigma -R_{\widetilde{\sigma}}\|_{-1,1}.
\]
Letting $\delta$ tends to zero, we get
\[
\max \left(\|\gamma _1\|_{C(\overline{\Omega})},\|\gamma _2\|_{C(\overline{\Omega})}\right)\le 2a _1^2 \| R_\sigma -R_{\widetilde{\sigma}}\|_{-1,1}.
\]

On the other hand, we can proceed as in the proof of Lemma \ref{lemma1} in order to get 
\[
\frac{1}{2}a_0^2 \| R_\sigma -R_{\widetilde{\sigma}}\|_{-1,1}\le \max \left(\|\gamma _1\|_{C(\overline{\Omega})},\|\gamma _2\|_{C(\overline{\Omega})}\right).
\]
In other words, we established the following two-sided estimate
\[
\frac{1}{2}a_0^2 \| R_\sigma -R_{\widetilde{\sigma}}\|_{-1,1}\le \max \left(\|\gamma _1\|_{C(\overline{\Omega})},\|\gamma _2\|_{C(\overline{\Omega})}\right)\le 2a _1^2 \| R_\sigma -R_{\widetilde{\sigma}}\|_{-1,1}.
\]

(ii) The same arguments applies in the any space dimension for continuous anisotropic diagonal conductivities of the form
 \[
 \sigma =\mbox{diag}(\sigma _1,\ldots \sigma _n).
 \]
 
(iii) The case of general symmetric continuous conductivities $\sigma$ can be handled by an extra diagonalisation argument that can be performed at each point $x_0$ in $\Omega$.

(iv) Handling the more general case of measurable matrix valued diffusivities requires significant extra work (\cite{cohen}).

 }
\end{remark}

\section{Neumann and Robin boundary conditions}

We explain briefly how Theorem \ref{theorem1} can be adapted to both Neumann and Robin BVP's. 

\subsection{The Neumann case}

For $\sigma \in \Sigma$, define  $A_\sigma ^N:H^1(\Omega )\rightarrow (H^1(\Omega ))'$ by
\[
\langle A_\sigma ^N u,v \rangle := \int_\Omega \sigma \nabla u \cdot \nabla v dx+\int_\Omega u v dx  ,\;\; u,v \in H^1(\Omega ),
\]
where $\langle \cdot ,\cdot \rangle$ is the duality pairing between $(H^1(\Omega ))'$ and $H^1(\Omega )$. 

\smallskip
Clearly $A_\sigma ^N$ is bounded and, with $ \underline{\sigma}_1=\max (\sigma_1,1)$,
\[
\|A_\sigma ^Nu\|_{(H^1(\Omega ))'}\le \underline{\sigma}_1\|u\|_{H^1(\Omega )},\;\; u\in H^1(\Omega ).
\]

We claim that $A_\sigma^N$ is invertible. Indeed, if $f\in (H^1(\Omega ))'$,  we get by applying Lax-Milgram's lemma that the variational problem
\begin{equation}\label{4.2}
\int_\Omega \sigma \nabla u_\sigma \cdot \nabla v dx+\int_\Omega u_\sigma v dx =\langle f,v \rangle ,\;\; v \in H^1(\Omega ).
\end{equation}
has a unique solution $u_\sigma\in H^1(\Omega )$.

Taking $v=u_\sigma$ in \eqref{4.2}, we get in a straightforward manner that
\begin{equation}\label{4.3}
\|u_\sigma \|_{H^1(\Omega )} \le \underline{\sigma}_0^{-1}\|f\|_{(H^1(\Omega ))'},\;\; \mbox{with}\; \underline{\sigma}_0=\min (\sigma _0,1).
\end{equation}

As a consequence of \eqref{4.2}, $A_\sigma ^Nu_\sigma =f$. Thus $A_\sigma^N$ has a bounded inverse 
\[
R_\sigma^N :=(A_\sigma^N)^{-1}:(H^1(\Omega ))'\rightarrow H^1(\Omega ).
\]
This operator is nothing but the resolvent of the operator $-\mbox{div}(\sigma \nabla \cdot )+1$ under the Neumann boundary condition. When $\Omega$ and $u$ are sufficiently smooth, this boundary condition can be written as  $
\partial_\nu u  =0$ on $\partial \Omega$, where $\partial _\nu =\nu \cdot \nabla$ with $\nu $ the exterior normal unit normal vector field on $\Gamma$.

\smallskip
As
\[
\langle  (A_\sigma ^N -A_{\widetilde{\sigma}}^N)u, v \rangle = \int_\Omega (\sigma -\widetilde{\sigma})\nabla u \cdot \nabla v dx,\;\; u,v\in H^1(\Omega ),
\]
we can mimic the proof  of Theorem \ref{theorem1} in order to get

\[
\|\sigma -\widetilde{\sigma}\|_{L^\infty (\Omega )}\le \underline{\sigma}_1^2\|R_\sigma ^N-R_{\widetilde{\sigma}}^N\|_{-1,1} .
\]

On the other hand,  we have, similarly to Lemma \ref{lemma1},
\[
\underline{\sigma}_0^2\|R_\sigma ^N-R_{\widetilde{\sigma}}^N\|_{-1,1}\le \|\sigma -\widetilde{\sigma}\|_{L^\infty (\Omega )}.
\]

In other words, we proved
\[
\underline{\sigma}_0^2\|R_\sigma ^N-R_{\widetilde{\sigma}}^N\|_{-1,1}\le \|\sigma -\widetilde{\sigma}\|_{L^\infty (\Omega )}\le\underline{\sigma}_1^2 \|\sigma -\widetilde{\sigma}\|_{L^\infty (\Omega )}.
\]

\subsection{The Robin case}

In the present subsection we assume that $\Omega$ has Lipschitz boundary $\Gamma$.

\smallskip
We examine the case of a BVP with a Robin boundary condition. To this end, pick $\beta \in L^\infty (\Gamma )$ so that $\beta \ge 0$ and $\beta \ge \beta _0$ on a measurable subset $\Gamma _0$ of $\Gamma$ of positive measure, where $\beta_0 >0$ is some constant. Consider the Robin BVP
\begin{equation}\label{eqr}
-\mbox{div}(\sigma \nabla u)=f\; \mbox{in}\; \Omega \;\; \mbox{and}\;\; \sigma \partial _\nu u+\beta u=0\; \mbox{on}\;\Gamma.
\end{equation}

\smallskip
If $\sigma \in \Sigma$, define $A_\sigma ^R:H^1(\Omega )\rightarrow (H^1(\Omega ))'$ by
\[
\langle A_\sigma ^R u,v \rangle := \int_\Omega \sigma \nabla u \cdot \nabla v dx+\int_{\Gamma}\beta u v dS(x)  ,\;\; u,v \in H^1(\Omega ).
\]
In the sequel we equip $H^1(\Omega )$ with the norm
\begin{equation}\label{r2}
\|u\|_{H^1(\Omega)}=\left(\|\nabla u\|^2_{L^2(\Omega )^n}+\|u\|_{L^2(\Gamma _0)}^2\right)^{1/2}.
\end{equation}

It is not hard to check that $A_\sigma ^R$ is bounded and
\[
\|A_\sigma ^Ru\|_{(H^1(\Omega ))'}\le \underline{\sigma}_1\|u\|_{H^1(\Omega )},\;\; u\in H^1(\Omega ),\;\; \mbox{with}\;  \underline{\sigma}_1=\max (\sigma_1,\kappa \|\beta \|_{L^\infty (\Gamma)}),
\]
where $\kappa$ is the norm of the trace operator $u\in H^1(\Omega )\rightarrow u_{|\Gamma}\in L^2(\Gamma )$ when $H^1(\Omega )$ is endowed with the norm \eqref{r2}.

Similarly to the Neumann case, we show that $A_\sigma ^R$ is invertible and we calculate its inverse. To do that we consider the bilinear form
\[
a(u,v)=\int_\Omega \sigma \nabla u\cdot \nabla vdx+\int_{\Gamma}\beta uvdS(x),\;\; u,v\in H^1(\Omega ).
\]
One can check that $u\rightarrow a(u,u)$ defines a norm on $H^1(\Omega )$ equivalent to the usual one on $H^1(\Omega )$. Let $f\in (H^1(\Omega ))'$. Then according to Riesz's representation theorem, there exists a unique $u_\sigma \in H^1(\Omega )$ satisfying
\begin{equation}\label{r1}
a(u_\sigma ,\psi )= \int_\Omega \sigma \nabla u_\sigma \cdot \nabla \psi dx+\int_{\Gamma}\beta u_\sigma \psi dS(x)=\langle f,\psi \rangle ,\;\;\psi \in H^1(\Omega ).
\end{equation}
Note that $u_\sigma$ is nothing but the variational solution of the BVP \eqref{eqr}.

\smallskip
From \eqref{r1}, we easily get
\[
\|u_\sigma \|_{H^1(\Omega )}\le \underline{\sigma}_0\|f\|_{(H^1(\Omega ))'},\;\; \mbox{with}\; \underline{\sigma}_0=\min (\sigma _0,\beta _0).
\]
Consequently, $A_\sigma ^R$ possesses a bounded inverse $R_\sigma ^R=(A_\sigma ^R)^{-1}:(H^1(\Omega ))'\rightarrow H^1(\Omega )$ defined by $R_\sigma f:=u_\sigma$ for $f\in (H^1(\Omega))'$.

\smallskip
Concerning the inverse problem for Robin boundary conditions,  starting  from
\[
\langle  (A_\sigma ^R -A_{\widetilde{\sigma}}^R)u, v \rangle _{-1,1}= \int_\Omega (\sigma -\widetilde{\sigma})\nabla u \cdot \nabla v dx,\;\; u,v\in H^1(\Omega ),
\]
we get similarly to the Neumann case
\[
\underline{\sigma}_0^2\|R_\sigma -R_{\widetilde{\sigma}}\|_{-1,1} \le \|\sigma -\widetilde{\sigma}\|_{L^\infty (\Omega )}\le \underline{\sigma}_1^2\|R_\sigma -R_{\widetilde{\sigma}}\|_{-1,1},
\]

\section{Non homogeneous BVP's}

In this section $\Omega$ is a $C^2$-smooth bounded domain of $\mathbb{R}^n$, $n\ge 2$, diffeomorphic to the unit ball of $\mathbb{R}^n$. Its boundary is denoted again by $\Gamma$.

\smallskip
Let $\sigma \in \Sigma$. For $g\in H^{\frac{1}{2}}(\Gamma )$, we denote by $u_\sigma\in H^1(\Omega )$ the unique weak solution of the BVP
 \[
 \mbox{div}(\sigma \nabla u)=0\; \mbox{in}\; \Omega \;\; \mbox{and}\;\; u=g\; \mbox{on}\; \Gamma .
 \]
Let $G\in H^1(\Omega )$ so that $G=g$ on $\Gamma$ and $\|G\|_{H^1(\Omega )}=\|g\|_{H^{\frac{1}{2}}(\Gamma )}$, where we identified $H^{\frac{1}{2}}(\Gamma )$ to the quotient space $H^1(\Omega )/H_0^1(\Omega )$. Then $f=-\mbox{div}(\sigma \nabla G)\in H^{-1}(\Omega )$ and 
\[
\| f\|_{H^{-1}(\Omega )}\le \sigma _1\|G\|_{H^1(\Omega )}=\sigma _1\|g\|_{H^{\frac{1}{2}}(\Gamma )}.
\]
On the other hand, it is straightforward to check that $u_\sigma =G+R_\sigma f$, $R_\sigma$ being the Dirichlet resolvent defined above. Therefore
\begin{equation*}
\|u_\sigma \|_{H^1(\Omega )}\le \|G\|_{H^1(\Omega )}+\sigma_0^{-1}\|f\|_{H^{-1}(\Omega )}
\le (1+\sigma_0^{-1}\sigma _1)\|g\|_{H^{\frac{1}{2}}(\Gamma )}.
\end{equation*}
Then $\Lambda_\sigma$ given by $\Lambda_\sigma g:=u_\sigma$ defines a bounded operator from $H^{\frac{1}{2}}(\Gamma )$ into $H^1(\Omega )$ and \[\|\Lambda_\sigma \|_{\frac{1}{2},1}\le 1+\sigma_0^{-1}\sigma _1.\] Here and in the sequel $\|\cdot \|_{\frac{1}{2},1}$  denotes the norm in $\mathscr{B}(H^{\frac{1}{2}}(\Gamma ),H^1(\Omega ))$.

Fix $\overline{g}\in C^2(\Gamma )$ so that  
\[
\Gamma _-=\{ x\in \Gamma ;\; \overline{g}(x)=\min \overline{g}\}\quad  \Gamma _+=\{ x\in \Gamma ;\; \overline{g}(x)=\max \overline{g} \}
\]
are nonempty and connected, and the following condition is fulfilled: there exists a continuous strictly increasing function  $\psi :[0,\infty )\rightarrow [0,\infty )$ with $\psi (0)=0$ and $\rho _0>0$ so that, for any $0<\rho \le \rho_0$,
\[
|\nabla _\tau \overline{g}|\ge \psi (\rho ),\;\mbox{on}\; \{x\in \Gamma ;\; \mbox{dist}(x,\Gamma _-\cup \Gamma _+)\ge \rho \}.
\]
where $\nabla _\tau$ denotes the tangential gradient.

\smallskip
Such a function is called quantitatively unimodal in \cite{ADFV}. 

\smallskip
We point out that the existence of a quantitatively unimodal function is guaranteed by the assumption that $\Omega$ is diffeomorphic to the unit ball.

\smallskip
For $\sigma _1>\sigma_0$, define
\[
\mathcal{E} =\{\sigma \in W^{1,\infty}(\Omega );\; \sigma_0 \le \sigma \; \mbox{and}\; \|\sigma \|_{W^{1,\infty}(\Omega )}\le \sigma _1\}.
\]

\begin{theorem}\label{sta1}
{\rm (\cite[Theorem 3.5]{ADFV})} There exist two constants $C>0$ and $\gamma >0$, that can depend on $\Omega$, $\mathcal{E}$ and $\overline{g}$, so that
\[
\| \sigma -\widetilde{\sigma}\|_{L^\infty(\Omega )}\le C\|\Lambda_\sigma \overline{g}-\Lambda_{\widetilde{\sigma}} \overline{g}\|_{L^2(\Omega )}^\gamma ,\;\; \sigma ,\widetilde{\sigma}\in \mathcal{E}_0,
\]
where $\mathcal{E}_0=\{ \sigma \in \mathcal{E};\; \sigma =\overline{\sigma}\; \mbox{on}\; \Gamma\}$, for some fixed $\overline{\sigma}\in \mathcal{E}$. 
\end{theorem}

This result is essential in the stability issue of the problem of determining the  conductivity coefficient from two attenuated energy densities obtained from well chosen two illuminations. This problem is related to the so-called qualitative photo-acoustic tomography. We refer to  \cite{ADFV} and the references therein for more details on this topic.

\smallskip
As a consequence of Theorem \ref{sta1} we readily obtain:

\begin{corollary}\label{sta2}
There exist two constants $C>0$ and $\gamma >0$, that can depend on $\Omega$ and $\mathcal{E}$, so that 
\[
\| \sigma -\widetilde{\sigma}\|_{L^\infty(\Omega )}\le C\|\Lambda_\sigma -\Lambda_{\widetilde{\sigma}} \|_{\frac{1}{2},1}^\gamma ,\;\; \sigma ,\widetilde{\sigma}\in \mathcal{E}_0,
\]
where $\mathcal{E}_0$ is as in the preceding theorem.
\end{corollary}

This result can be interpreted as a  H\"older stability estimate on the determination of $\sigma$ from $\Lambda_\sigma$.


\begin{remark}
\rm
Denote the lifting operator $g\rightarrow G$, defined above, by $E$ and, for $\sigma \in \mathcal{E}$, consider the operator $L_\sigma$ given by
\[
L_\sigma : F\in H^1(\Omega )\mapsto L_\sigma F={\rm div}(\sigma \nabla F)\in H^{-1}(\Omega ).
\]
Then one can check in a straightforward manner that $\Lambda_\sigma =E+R_\sigma L_\sigma E$. Therefore, the mapping 
\[
\sigma \in \mathcal{E}\mapsto \Lambda_\sigma \in \mathscr{B}(H^{\frac{1}{2}}(\Gamma ),H^1(\Omega ))
\]
is Lipschitz continuous. Whence, with reference to  Corollary \ref{sta2}, we get, for some constants $c>0$ and $C>0$,
\[
\| \sigma -\widetilde{\sigma}\|_{L^\infty(\Omega )}\le Cc^\gamma \| \sigma -\widetilde{\sigma}\|_{L^\infty(\Omega )}^\gamma ,\;\; \sigma ,\widetilde{\sigma}\in \mathcal{E}_0,
\]
or equivalently
\[
c\| \sigma -\widetilde{\sigma}\|_{L^\infty(\Omega )}\le Cc\left(c \| \sigma -\widetilde{\sigma}\|_{L^\infty(\Omega )}\right)^\gamma ,\;\; \sigma ,\widetilde{\sigma}\in \mathcal{E}_0.
\]

As $\| \sigma -\widetilde{\sigma}\|_{L^\infty(\Omega )}$ can be chosen arbitrarily small, we conclude that $\gamma \le1$.

We do not know whether, actually, $\gamma <1$ or not. Explicit computations can be carried out in the one-dimensional case, very much as in the next section. But  Theorem \ref{sta1}, which genuinely of multi-dimensional nature,  fails in this case since two different diffusivities, one multiple of the other, cannot be distinguished from boundary values. 
\end{remark}


\section{The  one-dimensional case}

\subsection{An explicit representation formula}
For the sake of simplicity, we limit our analysis to a BVP with mixed boundary conditions. Specifically, we consider the BVP
\begin{equation}\label{6.1}
-(\sigma (x)u_x)_x=f\; \mbox{in}\; (0,1),\;\; u_x(0)=0\;\; \mbox{and}\;\; u(1)=0.
\end{equation}

Let $\sigma _0 < \sigma _1 $ be two positive constants and $$\Sigma ^0=\{\sigma \in L^\infty (0,1);\; 0<\sigma _0\le \sigma \le \sigma _1 \; \mbox{a.e. in}\; (0,1)\}$$ and $$H=\{u\in H^1(0,1);\; u(1)=0\}.$$ It is a classical result that $u\in H\rightarrow \|u_x\|_{L^2(0,1)}$ defines a norm on $H$ which is equivalent to the norm induced by the usual norm on $H^1(0,1)$. In the sequel $H$ is equipped with this norm.

\smallskip
According to Lax-Milgram's lemma or Riesz's representation theorem, for each $f\in H'$, there exists a unique $u=u_\sigma \in H$ so that
\[
\int_0^1 \sigma (x)u_x(x)v_x(x) dx=\langle f,v\rangle,\;\; \mbox{for any}\; v\in H,
\]
where $\langle \cdot,\cdot\rangle$ is the duality pairing between $H$ and its dual $H'$. Note that $u_\sigma$ is nothing but the variational solution of the BVP \eqref{6.1}.

\smallskip
Therefore $R_\sigma :f\in H'\rightarrow u_\sigma \in H$ defines a bounded operator with
\[
\|R_\sigma f\|_H\le \sigma _0^{-1}\|f\|_{H'}.
\]

\smallskip
Pick $f\in L^2(0,1)$ and set
\[
v(x)=\int_x^1\frac{1}{\sigma (t)}\int_0^tf(s)ds dt,\;\; x\in [0,1].
\]
Clearly $v$ is absolutely continuous, $v(1)=0$ and 
\begin{equation}\label{6.2}
v_x(x)=-\frac{1}{\sigma (x)}\int_0^x f(t)dt\;\; \mbox{a.e.}\; (0,1).
\end{equation}
On the other hand, if $w\in H$, we get by applying Green's formula
\[
\int_0^1\sigma (x)v_x(x)w_x(x)=-\int_0^1w_x(x)\left(\int_0^xf(t)dt\right)dx=\int_0^1w(x)f(x)dx.
\]
In other words, $v=R_\sigma f$.

\subsection{Lipschitz stability}

In view of the explicit representation formula above   it is convenient to introduce the space $W^{-1,1}(0,1)$, the closure of $C_0^\infty (0,1)$ for the norm

\[
\|f \|_{W^{-1,1}(0,1)} = \left\| \int_0^xf(t)dt \right\|_{L^1((0,1))}.
\]

The resolvent operators, according to the explicit representation formula above, can be represented as
\begin{equation}\label{6.1.1}
T_mf(x)=m(x)\int_0^xf(t)dt,\;\; \mbox{a.e.}\;\; x\in (0,1),
\end{equation}
where $$m=\frac{1}{\sigma}.$$

Given $m\in L^\infty (0,1)$, these operators can be naturally understood in the functional setting  of linear bounded operators, $T_m:W^{-1, 1}(0,1)\rightarrow L^1(0,1)$. We denote by $\mathscr{B}(W^{-1,1}(0, 1), L^1(0,1))$ this Banach space and by $\|\cdot \|_{-1,1}$ its norm.

The following holds:

\begin{lemma}\label{lemma6.1}
Let $m\in L^\infty (0,1)$ and $T_m:W^{-1, 1}(0,1)\rightarrow L^1(0,1)$ as in (\ref{6.1.1}).
Then \begin{equation}
\|T_m\|_{-1,1}=\|m\|_{L^\infty (0,1)}.
\end{equation}

\end{lemma}

\begin{proof}
Firstly, it is straightforward to check that $$\|T_m\|_{-1,1}\le \|m\|_{L^\infty ((0,1))}.$$ 

The reverse inequality can be easily derived as in Lemma \ref{lemma0.1}, by taking a sequence $f_\varepsilon$ so that the corresponding primitives
$$
F_\varepsilon = \int_0^x f_\varepsilon(t) dt,
$$
constitute an approximation of the identity around each $x_0 \in (0, 1)$.
 \end{proof}


Given two diffusivity coefficients $\sigma ,\widetilde{\sigma}\in \Sigma ^0$, formula \eqref{6.2} yields
\[
\left( R_\sigma f-R_{\widetilde{\sigma}}f\right)_x=\left(\frac{1}{\widetilde{\sigma}(x)}-\frac{1}{\sigma(x)}\right)\int_0^x f(t)dt\;\; \mbox{a.e.}\; (0,1).
\]
In view of Lemma \ref{lemma6.1} it is natural to analyze the norms of these resolvent operators and their distances in the norm
$$
|| R_\sigma  ||_{*}  = || T_{1/\sigma} ||_{-1, 1}.
$$
As a consequence of Lemma \ref{lemma6.1}
$$
|| R_\sigma -R_{\widetilde{\sigma}} ||_{*}  = \left\| \frac{1}{\widetilde{\sigma}(x)}-\frac{1}{\sigma(x)}\right\|_{L^\infty((0, 1))}.
$$

Obviously, using the uniform upper and lower bounds on the coefficients, this also allows to get estimates in terms of the $L^\infty(0,1)$-distances between coefficients:
\begin{equation}\label{identitynorms}
\sigma_1^{-2}|| \widetilde{\sigma}-\sigma ||_{L^\infty(0, 1)} \le || R_\sigma -R_{\widetilde{\sigma}} ||_{*}  \le \sigma_0^{-2} || \widetilde{\sigma}-\sigma ||_{L^\infty(0, 1)}.
\end{equation}

This is so because 
$$
\frac{\sigma - \widetilde{\sigma} }{\sigma  \widetilde{\sigma} }= \frac{1}{\widetilde{\sigma}}-\frac{1}{\sigma}.
$$

\subsection{Distance to a subspace}\label{surrogate}
As we shall see in the following section, in the application of greedy algorithms we need to further develop the computations above to achieve precise Lipschitz stability estimates  on the distance from one given resolvent to the subspace generated by a finite number of others.

To do this, we consider a distinguished coefficient that we denote by $\tau(x)$ and $N\ge 2$ others, $\sigma_1(x), \cdots, \sigma_N(x)$, and denote the corresponding resolvents by $R_\tau$ and $R_1, \cdots, R_N$, respectively.

As a straightforward application of  identity (\ref{6.2}), we have
\begin{equation}\label{distinv}
\left( R_\tau f-\sum_{i=1}^N a_i R_{i}f\right)_x=\left(\sum_{i=1}^N \frac{a_i}{\sigma_i(x)}-\frac{1}{\tau(x)}\right)\int_0^x f(t)dt\;\; \mbox{a.e.}\; (0,1)
\end{equation}
that yields the representation of the difference of a resolvent with respect to the linear combination of a finite number of others.

Arguing as above we can conclude that
\begin{equation}\label{1dssurrogate}
\left \| R_\tau -\sum_{i=1}^N a_i R_{i} \right \|_* = \left \| \sum_{i=1}^N \frac{a_i}{\sigma_i(x)}-\frac{1}{\tau(x)} \right \|_{L^\infty((0, 1))}.
\end{equation}

In other words, we have shown that the $L^\infty$-distance between inverses of coefficients, yields an adequate surrogate for the distance between the resolvents : 
\begin{equation}\label{1dsurrogate}
 \mbox{dist}_{*}(R_\tau, \mbox{span} [R_{i}, 1\le i \le N] ) = \mbox{dist}_{L^\infty(0, 1)} \left(\frac{1}{\tau(x)}, \mbox{span} \left[ \frac{1}{\sigma_i(x)}, 1\le i \le N \right] \right).
\end{equation}

Here the distance $\mbox{dist}_{*}$ stands for the one given in terms of the $\| \cdot \|_*$-norm.

\smallskip

Note that the methods of the previous sections do not allow to achieve similar results in the multi-dimensional case. In particular, the analysis of the one-dimensional case shows that when dealing with the distance between a resolvent to the span of several others one has to analyze nonlinear expressions involving the diffusivity coefficients. 

\section{Application to greedy algorithms for parameter depending elliptic equations}

In recent years there has been a significant body of literature developed on greedy methods to approximate parameter-dependent elliptic problems of the form
\begin{equation}\label{ellipticparameter}
-\mbox{div}(\sigma  (x,\mu ) \nabla u)=f\; \mbox{in}\; \Omega , \quad u=0\; \mbox{on}\; \Gamma.
\end{equation}
We refer for instance to \cite{BCDDPW, Buffa12, CDV, CD15, DV15, DPW13}.

\smallskip
Roughly, the problem has been formulated and addressed  as follows. 

\smallskip
Assume that $\sigma (x, \mu) \in \Sigma$ (with $\Sigma$ as in previous sections) depends on a multi-parameter $\mu$ living on a compact set $\mathscr{K}$ of $\mathbb{R}^d$ with $d \ge 1$ finite. We denote by  $\mathscr{S}$ the parametrized set of coefficients $\sigma(x, \mu )\in \Sigma$ for all value of $\mu$. 

\smallskip
Given a fixed $f\in H^{-1}(\Omega)$ and solving  (\ref{ellipticparameter}) we get the set $\mathscr{U}$ of the corresponding solutions $u(x, \mu) \in H^1_0(\Omega), \mu \in \mathscr{K}$.
This set  inherits the regularity  of  the coefficients $\sigma(x, \mu ) \in \Sigma$ in its dependence with respect to $\mu$. 

\smallskip
The question that has been considered so far  consists in identifying the most distinguished values of the parameter $\mu$ to better approximate the set of solutions $\mathscr{U}$, for that specific given right hand side term 
 $f\in H^{-1}(\Omega)$. This has been done applying (weak) greedy algorithms obtaining optimal approximation rates for $\mathscr{U}$. But, proceeding that way, the sequence of most relevant snapshots $\mu_n$ that the algorithm gives depends on the right hand side term $f \in H^{-1}(\Omega)$ and 
different right hand side terms $f$ lead to different choices of the snapshots of $\mu$.

\smallskip
 Theorem \ref{theorem1} was developed in an attempt  to apply the same methods independently of the specific value of the right hand side term $f \in H^{-1}(\Omega)$. However, this program can only be achieved so far in the one-dimensional case where our analysis was much more complete.

\smallskip

 For this to be done one needs to deal with the set $\mathscr{R}$ of resolvent operators $R(\mu)$ in $\mathscr{B}(H^{-1}(\Omega ),H_0^1(\Omega ))$ for all $\mu \in \mathscr{K}$ which inherit the continuity and regularity properties of the coefficients $\sigma(x, \mu)\in \Sigma$ in its dependence with respect to $\mu$. For instance, if the map $\mu \in 
\mathscr{K} \to \sigma(x, \mu)\in L^\infty(\Omega)$ is continuous, the same occurs for the map $\mu \in 
\mathscr{K} \to R(\mu) \in \mathscr{B}(H^{-1}(\Omega ),H_0^1(\Omega ))$. The same can be said about the $C^k$, $C^\infty$ or analytic dependence. On the other hand, the compactness of the set 
$\mathscr{K}$ together with the continuous dependence on $\mu$ ensures the compactness of $\mathscr{R}$ in $\mathscr{B}(H^{-1}(\Omega ),H_0^1(\Omega ))$.

\smallskip
The goal is then to approximate the  compact set of resolvents $\mathscr{R}$ of the Banach space $\mathscr{B}(H^{-1}(\Omega ),H_0^1(\Omega ))$ by a sequence of finite dimensional subspaces $V_n$ of dimension $n\ge 1$.  The weak greedy algorithms yield the subspaces $V_n$ that approximate the set $\mathscr{R}$ in the best possible manner, in the sense of the Kolmogorov $n$-width.
The subspaces $V_n$ are defined as the span of the most distinguished resolvent operators $R(\mu_1),\ldots , R(\mu_n)$ with $\mu_1,...,\mu_n$ chosen as follows.

\smallskip
Fix a constant $\gamma \in (0, 1)$. 
Choose  $\mu_1 \in \mathscr{K}$ such that 
\begin{equation}
\label{u_1}
||R(\mu_1) ||_{-1,1} \geq \gamma  \max_{\mu \in \mathscr{K}} ||R(\mu) ||_{-1,1} .
\end{equation}
We then proceed in a recursive manner. Having found $\mu_1,\ldots ,\mu_n$, denote $V_n={\rm span} \{R(\mu_1), \ldots, R(\mu_n)\} $ and
choose the next element $\mu_{n+1}$ such that
\begin{equation}
\label{general-step}
{\rm dist}(R(\mu_{n+1}), V_n) \geq \gamma  \max_{\mu \in \mathscr{K}} {\rm dist}(R(\mu), V_n)\,.
\end{equation}

From the previous existing theory (see \cite{DPW13} and \cite{D15}) this algorithm is well known to yield nearly optimal approximation rates of the parameterized set of resolvents $\mathscr{R}$. As indicated in the previous references because, now, we are in a Banach space setting, $\mathscr{B}(H^{-1}(\Omega ),H_0^1(\Omega ))$, a loss of the order $1/2$ on the approximation rate of the Kolmogorov width has to be expected.

\smallskip
However, the difficulty on its implementation consists in computing, in each step, the distance ${\rm dist}(R(\mu), V_n)$. This would require, in particular, computing $R(\mu)$ for all values of $\mu$ and this is unfeasible and, precisely, what we want to avoid.

\smallskip
But the existing theory has also developed a means of bypassing this difficulty. In fact, it is well known that the same algorithm yields optimal approximation rates if it is implemented, as indicated above, but  replacing ${\rm dist}(R(\mu), V_n)$ by a ``surrogate", i.e. a different distance function, easier to be computed, and giving a uniform bound from below for the true distance ${\rm dist}(R(\mu), V_n)$.

\smallskip
Theorem \ref{theorem1} is a first attempt in that direction, ensuring that the $L^1$-distance between two  coefficients  provides a lower bound on the distance between the resolvents.
But the issue of finding true surrogates for the distance of a resolvent to the subspace generated by a finite number of others is open in the general multi-dimensional case. The results in Subsection \ref{surrogate} yield such a surrogate in dimension $n=1$, see (\ref{1dsurrogate}).

\smallskip
Accordingly, in one space dimension, $n=1$, the  implementation of the weak greedy algorithm for the approximation of the parameter set of resolvents can be done as above, by modifying the recursive step as follows: Fix some $0 <\gamma < 1$. Having found $\mu_1, \ldots , \mu_n$, with the corresponding diffusivity coefficients $\sigma_1(x),\ldots , \sigma_N(x)$, to
choose the next element $\mu_{n+1}$ such that the corresponding diffusivity coefficient $\sigma_{N+1}$ satisfies
\begin{align}\label{greedyL1}
{\rm dist}_{L^\infty(0, 1)} \left (\frac{1}{\sigma_{N+1}}, {\rm span} \left [ \frac{1}{\sigma_i}; i, \dots, N \right ] \right ) \geq \gamma  \max_{\mu \in \mathscr{K}}{\rm dist}_{L^\infty(0, 1)} \left (\frac{1}{\sigma(\mu)}, {\rm span} \left [ \frac{1}{\sigma_i}; i=1, \dots, N \right ] \right )\,.
\end{align}

The important consequence of this fact is that, for the identification of the most relevant parameter values $\mu_n$, we do not  need to solve the elliptic equation, but simply deal with the family of coefficients $\sigma(x, \mu)$, solving a by now classical $L^\infty$-minimisation problem in an approximated manner as indicated in (\ref{greedyL1}) by a multiplicative factor ($0<\gamma<1$). Once this choice of $\mu_n$ is done, this readily allows identifying the most relevant resolvent or elliptic problem,  for all values of the right hand side term $f\in H^{-1}(\Omega)$, contrarily to previous developments where the choice of these snapshots was $f$-dependent.

\smallskip

The choice of the parameters $\mu_n$ that we achieve in this manner is optimal from the point of view of the approximation of the resolvents and can then be applied to any $f\in H^{-1}(\Omega)$, as mentioned above. But, of course, for a specific value of $f\in H^{-1}(\Omega)$, the ad-hoc application of the weak greedy method will lead to better approximations. But for this to be done one has to afford  the cost of implementing the weak greedy method for each $f\in H^{-1}(\Omega)$ again and again. The advantage of the method developed in this paper is that it leads to uniform, robust approximation results, valid for all $f\in H^{-1}(\Omega)$ and can be used as a preconditioner to later use further greedy arguments, adapted to each right hand side term $f$.

\smallskip
In practice, given an arbitrary value of $\mu$, the resolvent can be identified with the corresponding multiplier $1/\sigma(\mu)$. Therefore, it can be approximated by a suitable linear combination of the weak-greedy offline choices $1/\sigma_i$. This gives an easy and computationally inexpensive way of approximating the resolvent associated to $\mu$.

\smallskip
Note however that this program was only fully developed in the one-dimensional case since the multi-dimensional analogue of the surrogate in terms of the coefficients is not known. In the next section we address, in the multi-dimensional context, the simpler problem of variables density functions.


\begin{remark}
{\rm
It is worthwhile mentioning that the program we carried out for the conductivity coefficient in the one dimensional case can be extended to the problem of identifying the density coefficient from the corresponding resolvent in an arbitrary dimension.

\smallskip
For sake of simplicity we assume in this remark that $\Omega$ is $C^{1,1}$-smooth.

\smallskip
Consider, for $f\in L^2(\Omega )$ and $\rho \in L^\infty (\Omega )$, the problem of finding $u\in H_0^1(\Omega )$ so that 
\[
-\Delta u =\rho f\;\; \mbox{in}\; \Omega .
\]
The corresponding variational formulation consists in seeking $u\in H_0^1(\Omega )$ satisfying
\begin{equation}\label{rem1}
\int_\Omega \nabla u\cdot \nabla vdx = \int_\Omega \rho f vdx ,\;\; v\in H_0^1(\Omega ).
\end{equation}
By Lax-Milgram lemma this problem has a unique solution $u_\rho :=R_\rho f$.

\smallskip
As $\Omega$ is $C^{1,1}$-smooth, $R_\rho$ maps $L^2(\Omega )$ into $\mathcal{H}=H_0^1(\Omega )\cap H^2(\Omega )$. When $\rho \equiv 1$, we denote $R_\rho$ by $R$. Then $R$ is nothing but the inverse of the bounded operator $A: \mathcal{H}\rightarrow L^2(\Omega )$ given by $Au=-\Delta u$.

 \smallskip
 Using once again Lebesgue's differentiation theorem, we can prove that $M_\rho$, the multiplication operator by $\rho \in L^\infty (\Omega)$, acting as an operator on $L^2(\Omega )$, satisfies $\|M_\rho \|:=\|M_\rho\|_{\mathscr{B}(L^2(\Omega ))}=\|\rho\|_{L^\infty (\Omega )}$.
 
\smallskip
Since $R_\rho =RM_\rho$ or equivalently $M_\rho = AR_\rho$, we derive
\[
\|R\|^{-1} \|R_\rho \|\le \| \rho \|_{L^\infty (\Omega )}\le \|A\| \|R_\rho\|.
\]

Fix $\rho_1,\ldots ,\rho_N\in L^\infty (\Omega )$, and let $\rho \in V_N=\mbox{span}\{\rho_1,\ldots \rho_N\}$ and $\widetilde{\rho}\in L^\infty (\Omega )$. In light of the linearity of the mapping $\rho \rightarrow R_\rho$ we get
\[
\|R\|^{-1} \|R_\rho -R_{\widetilde{\rho}} \|\le \| \rho -\widetilde{\rho}\|_{L^\infty (\Omega )}\le \|A\| \|R_\rho-R_{\widetilde{\rho}}\|.
\]
Accordingly 
\[
\mathbf{d}(R_{\widetilde{\rho}}, \mathbf{R}_N)=\mbox{dist}_{L^\infty (\Omega )}(\widetilde{\rho} ,V_N),
\]
where $\mathbf{R}_N=\mbox{span}\{R_{\rho_1},\ldots ,R_{\rho_N}\}$
and
 \[
\mathbf{d}(R_\rho , R_{\widetilde{\rho}})=\| \rho -\widetilde{\rho}\|_{L^\infty (\Omega )}.
\]
 yields an appropriate surrogate between resolvents.
 
 \smallskip
 This allows the full application of the weak greedy algorithm described in the previous section in this case of multi-dimensional density dependent elliptic equations.
 }
 \end{remark}

\section{Extensions and further comments.}

The results of this paper constitute a first contribution on a topic that is rich in open interesting problems.

We mention here some of them:
\begin{itemize}

\item {\it Surrogates in the multi-dimensional case}. In Subsection \ref{surrogate} we have found surrogates for the elliptic problem with variable diffusivity  in dimension $n=1$. This problem is totally open in the multi-dimensional case.
The case of the variable density was solved in the previous section.

\item {\it Elliptic matrices.} In dimensions $n \ge 2$ the same problems can be formulated in the context of elliptic problems involving coefficients $\sigma_{ij}(x, \mu)$, i. e. to equations of the form
\[
-\sum_j \partial_j(\sigma_{ij}(x, \mu) \partial_i u)=f.
\]
Of course, the problem is much more complex in this case since there is no a sole coefficient $\sigma$ to be identified but rather all the family $\sigma_{ij}$ with $i, j=1,...,N$.

\item {\it Elliptic systems.}  The same problems arise also in the context of elliptic systems such as, for instance, the system of elasticity.

\item {\it Evolution equations.} The problems addressed in the present paper make also sense for evolution problems and, in particular, parabolic, hyperbolic and Schr\"odinger equations.

Let us consider for instance the heat equation:
\begin{equation}\label{heat}
\left\{\begin{aligned}
u_t-\mbox{div}(\sigma \nabla u)=0\; \mbox{in}\; \Omega \times (0, \infty) \\
 u=0\; \mbox{on}\;\Gamma \times (0, \infty)\\
 u(x, 0)= f(x) \; \mbox{in}\;\Omega.
 \end{aligned}\right.
\end{equation}

The same questions we have addressed in the elliptic context arise also in the parabolic one. In this case, the question can be formulated as follows: {\it Does the resolvent $f \in L^2(\Omega) \to C([0, \infty); L^2(\Omega))$ determine the diffusivity coefficient in an unique manner? Is the map from resolvent to diffusion coefficient Lipschitz in suitable norms?}

The way the question has been formulated is easy to solve. In fact it is sufficient to observe that the elliptic equation is subordinated to the parabolic one, so that, the time integral of the parabolic solution, namely,
$$
v(x)=\int_0^\infty u(x, t) dt,
$$
solves the elliptic equation
\begin{equation}\label{ellip-subo}
\left\{\begin{aligned}
-\mbox{div}(\sigma \nabla v)=f\; \mbox{in}\; \Omega \\
 v=0\; \mbox{on}\;\Gamma.
  \end{aligned}\right.
\end{equation}
This can be easily seen integrating the parabolic equation in time and using the fact that the solutions of the heat equation tend to zero as $t \to \infty$.

Therefore, a full knowledge of the parabolic resolvent yields, in particular, fully, the elliptic resolvent as well. According to the previous results in this paper the diffusivity coefficient can be determined, and the dependence  is Lipschitz. Our comments on the explicit representation of solutions in $1-d$ and their possible use for the development of greedy algorithms apply as well.

This simple observation however raises many other interesting problems: Can the same Lipschitz identification result be achieved if the parabolic solution is only known in a finite time interval $[0, T]$? What about diffusivity coefficients depending also on time $\sigma=\sigma(x, t)$? What happens when, rather than the solution for the initial value problem, one considers those with non-homogeneous source terms?

Similar questions arise for wave-like equations and the same answer can be expected if the models under consideration are dissipative. This allows to integrate the equations under consideration in the time-interval $(0, \infty)$ and employing the decay of solutions as $t\to \infty$. Of course, the expected behaviour is completely different in the absence of damping.

\item {\it Control problems.} Greedy and weak greedy methods have been implemented in the context of controllability of finite and infinite-dimensional Ordinary Differential Equations (ODEs) in \cite{Lazar}. But this has been done for fixed specific data to be controlled. It would be interesting to analyze whether the results of this paper can be extended to these controllability problems so to achieve approximations, independent of the data to be controlled. For this to be done one would need to adapt the Lipschitz stability estimate in (\ref{Lest}) to control problems, getting upper bounds on the distance between the generators of the semigroups in terms of the distance between the corresponding control maps. This is an issue that needs significant further investigation.
\end{itemize}

\medskip
\textbf{Acknowledgement.} We would like to think Giovanni Alessandrini (Trieste), Albert Cohen (Paris) and Martin Lazar (Dubrovnik) for their valuable comments.


\end{document}